\documentclass[a4paper]{article}

\usepackage[a4paper,top=3cm,bottom=2cm,left=3cm,right=3cm,marginparwidth=1.75cm]{geometry}

\usepackage{url}
\usepackage{amsmath, amsthm, amsfonts}
\usepackage{graphicx}
\usepackage[colorinlistoftodos]{todonotes}
\usepackage[colorlinks=true, allcolors=blue]{hyperref}
\usepackage{mathrsfs}

\newcommand{\call}[1]{\mathcal{#1}}
\newcommand{\double}[1]{\mathbb{#1}}
\newcommand{\C}{\double{C}}

\newcommand{\N}{\double{N}}
\newcommand{\R}{\double{R}}

\newcommand{\hol}{\call{O}}
\newcommand{\D}{\call{D}}
\newcommand{\mm}{\mathfrak{m}}

\newcommand{\what}[1]{\widehat{#1}}
\newcommand{\HHom}{\mathcal{H}om}

\DeclareMathOperator{\leftfun}{left}
\DeclareMathOperator{\Der}{Der}

\DeclareMathOperator{\rank}{rank}
\DeclareMathOperator{\Adj}{Adj}
\DeclareMathOperator{\tr}{tr}
\DeclareMathOperator{\Obj}{Obj}
\DeclareMathOperator{\Sing}{Sing}

\theoremstyle{plain}
\newtheorem{Th}{Theorem}[section]
\newtheorem{Prop}{Proposition}[section]
\newtheorem{Cor}{Corollary}[section]
\newtheorem{Lema}{Lemma}[section]
\newtheorem{Ex}{Example}[section]
\newtheorem{Conj}{Conjecture}[section]
\theoremstyle{definition}
\newtheorem{Def}{Definition}[section]
\newtheorem*{Rmk}{Remark}

\title{\textbf{On Euler-homogeneity for free divisors}}
\author{Abraham del Valle Rodríguez}
\date{}

\begin{document}
\maketitle

\begin{abstract}
In 2002, it was conjectured that a free divisor satisfying the so-called Logarithmic Comparison Theorem must be strongly Euler-homogeneous and it was proved for the two-dimensional case. Later, in 2006, it was shown that the conjecture is also true in dimension three, but, today, the answer for the general case remains unknown. In this paper, we use the decomposition of a singular derivation as the sum of a semisimple and a topologically nilpotent derivation that commute in order to deal with this problem. By showing that this decomposition preserves the property of being logarithmic, we are able to give alternative proofs for the low-dimensional known cases.

\end{abstract}

\vspace{0,3 cm}

\section{Introduction}

Let $X$ be a complex analytic manifold of dimension $n$ and $D \subset X$ be a divisor. Let $\Omega^\bullet_X(*D)$ be the complex of meromorphic differential forms with poles along $D$ and $\Omega^\bullet_X(\log D) \subset \Omega^\bullet_X(*D)$ the subcomplex of logarithmic differential forms. We say that the Logarithmic Comparison Theorem (LCT) holds for $D$ if the inclusion

$$ \Omega^\bullet_X(\log D) \hookrightarrow \Omega^\bullet_X(*D) $$

\noindent is a quasi-isomorphism.

\vspace{0,3 cm}

\noindent By Grothendieck's Comparison Theorem \cite[Theorem 2]{Groth}, if we write $U = X \setminus D$ and $j: U \hookrightarrow X$, then the morphism

$$ \Omega^\bullet_X(*D) \to Rj_* \C_U $$

\vspace{0,3 cm}

\noindent is a quasi-isomorphism. So, if LCT holds for $D$, then $\Omega^\bullet_X(\log D)$ computes the cohomology of its complement.  \\

\noindent We denote by $\Der_X(-\log D)$ the $\hol_X$-module of logarithmic derivations (or vector fields) along $D$, which is the dual of $\Omega^1(\log D)$. Locally, a derivation $\delta \in \Der_{X,x}$ belongs to $\Der_{X,x}(-\log D)$ if $\delta(f) \in \langle f \rangle$ for any reduced local equation $f$ of $D$ at $x$. In this case, we will say that $\delta$ is a logarithmic derivation for $f$.

\vspace{0,3 cm}

\begin{Def}
(See \cite{Saito}.)
$D$ is called a \emph{free} divisor if $\Der_X(-\log D)$ (or, equivalently, $\Omega^1_X(\log D)$) is a locally free $\hol_X$-module.
\end{Def}

\vspace{0,3 cm}

\noindent Let $x \in X$ and consider the stalk $\Der_{X,x}(-\log D)$. Given a set of derivations $S = \{\delta_1, \ldots, \delta_n\}$ in $\Der_{X,x}(-\log D)$, we call $A = (\delta_i(x_j))_{i,j}$ \emph{the Saito matrix with respect to $S$} (and to the coordinate system). Denoting by $\bar{\delta} = (\delta_1, \ldots, \delta_n)^t$ and $\bar{\partial} = (\partial_1, \ldots, \partial_n)^t$, we have that $\bar{\delta} = A \bar{\partial}$.

\vspace{0,3 cm}

\noindent If $D$ is free, then $\Der_{X,x}(-\log D)$ is free as an $\hol_{X,x}$-module for every $x \in X$. By Saito's criterion \cite[Theorem 1.8]{Saito}, the freeness of $\Der_{X,x}(-\log D)$ is equivalent to the existence of some $\delta_1, \ldots, \delta_n \in \Der_{X,x} (-\log D)$ such that if $f \in \hol_{X,x}$ is a reduced local equation around $x$, then there exists a unit $u \in \hol_{X,x}$ with $f = u \det(A)$, where $A$ is the Saito matrix with respect to $\delta_1, \ldots, \delta_n$. In this case, $\{\delta_1, \ldots, \delta_n\}$ is a basis of $\Der_{X,x}(-\log D)$ and we will also say that $f$ is \emph{free}. \\

\begin{Def}
\label{DefFEH}
A germ of holomorphic function $f \in \hol_{X,x}$ is called \emph{strongly Euler-homogeneous at $x$} if there exists a $\delta \in \mm_{X,x} \Der_{X,x}$ such that $\delta(f) = f$ (where $\mm_{X,x}$ denotes the maximal ideal of the local ring $\hol_{X,x}$). A divisor $D$ is called \emph{strongly Euler-homogeneous} if, for each $x \in D$, some (and hence, any) reduced local equation of $D$ at $x$ is strongly Euler-homogeneous.
\end{Def}

\vspace{0,3 cm}

\noindent In \cite{2002}, the authors proposed the following conjecture and proved it in the case $n = 2$: \\

\begin{Conj}
\label{ConjLCT}
If $D$ is a free divisor in a complex analytic manifold $X$ of dimension $n$ that satisfies the Logarithmic Comparison Theorem, then it is strongly Euler-homogeneous.
\end{Conj}

\vspace{0,3 cm}

\noindent Later, in 2006, M. Granger and M. Schulze proved in \cite[Theorem 1.6]{GS} that Conjecture \ref{ConjLCT} also holds for $n=3$. They did it by proving the so-called \emph{formal structure theorem}, by which they were able to find a basis of logarithmic derivations formed by diagonal and nilpotent derivations in a suitable system of formal coordinates that satisfy additional properties. \\

\noindent The main objective of this article is to give a new proof of the conjecture for these cases with the aim of better understanding the conjecture. We expect that this new perspective could provide us some hints to approach the conjecture in higher dimension. \\

\noindent Throughout this paper, we will deal with both formal and convergent objects. Some definitions for a convergent series can be extended in a natural way to formal ones, so Section \ref{SecPrelim} is destinated to present those definitions and see the relations between them.   \\

\noindent In Section \ref{SecSN} we give the definitions of semisimple and topologically nilpotent derivations, whose importance will become clear later on. These concepts, introduced by R. Gérard and A. Levelt in \cite{GL} generalize, and make intrinsic the ones given by K. Saito in \cite{Saito71}. They allow them to prove an analogue to the Jordan decomposition theorem, which states that a singular formal derivation can be decomposed in a unique way as the sum of two commuting derivations, where one of them is semisimple and the other is topologically nilpotent. Our main result in this section is that if a singular derivation is logarithmic, then so do its semisimple and topologically nilpotent parts, which allows us to give a basis of the logarithmic derivations in which each element has one of these properties.\\

\noindent In Section \ref{SecTr0} we show that if LCT holds for a free divisor $D$, then at least one of the logarithmic derivations must have non-zero trace. This was first noticed in \cite{2002}, but here we give an alternative and more conceptual proof based on $\D$-module theory. \\

\noindent Finally, in Section \ref{SecLCTEH}, we use the central result in Section \ref{SecSN} to show that every non-Euler derivation of a certain type of plane curve must be topologically nilpotent. As a corollary, using the result in Section \ref{SecTr0}, we get an alternative proof of the conjecture in dimension $n=2$. We also give an example that shows that this is no longer true in higher dimension. However, these results allow us, by reproducing part of the arguments given in \cite{GS}, to give another proof for the case $n=3$. \\

\noindent While proving the conjecture for $n=2$, the authors state in \cite{2002} that, for a certain kind of plane curves, there always exists a basis of the logarithmic derivations such that one of them has no linear part. Although this happens to be true, it seems to be not so immediate as they claim, so in Appendix \ref{Sec2002} we clarify the proof of this statement. We prove that there are only two possibilities and later, after some work, we are able to discard one of them, reaching the desired conclusion. \\

\noindent The author would like to thank his PhD advisors Luis Narváez Macarro and Alberto Castaño Domínguez for their support and guidance while writing this paper. He is supported by a Junta de Andalucía PIF fellowship num. PREDOC\_00485 and also partially supported by PID2020-114613GB-I00.

\section{Preliminaries}
\label{SecPrelim}

\noindent As we will see in Section \ref{SecSN}, the Jordan decomposition of a derivation requires to consider formal objects instead of just convergent ones. That is why we need to give some definitions for both convergent and formal series. We are interested in how these properties behave when a convergent series is seen as a formal one. \\

\noindent Let us introduce some relevant notation. Consider the local ring of convergent power series $\hol := \hol_{\C^n,0} = \C\{x_1, \ldots, x_n\}$ with $\mm = \langle x_1, \ldots, x_n \rangle$ its maximal ideal and its completion (with respect to the $\mm$-adic topology) $\what{\hol} = \C[[x_1, \ldots, x_n]]$, which is the ring of formal power series. \\

\noindent When working with free divisors, we will always consider reduced equations (i.e. with no repeated factors) in $\hol$. This property is preserved when they are seen as elements of $\what{\hol}$. \\

\begin{Prop}
\label{PropConvFormalRed}
    Let $f \in \hol$ be a reduced convergent power series. Then, it is also reduced as a formal power series.
\end{Prop}

\begin{proof}
    Since $f$ is reduced, the local ring $\hol/\langle f \rangle$ is reduced. The ring $\hol$ is excellent \cite[Theorem 102]{MatsumuraRosa}, so it is a Nagata ring \cite[Theorem 78]{MatsumuraRosa} and so does $\hol/\langle f \rangle$, since this property is preserved under quotients. But the completion of a local reduced Nagata ring is also reduced \cite[AC IX 4.4. Corollaire 1]{Bourbaki89}, so $\what{\hol}/\langle f \rangle$ is reduced, meaning $f$ is reduced as a formal power series. 
    
\end{proof}

\noindent Let us denote by $\Der$ the $\hol$-module of $\C$-derivations of $\hol$, which is free with a basis given by the partial derivatives $\{\partial_1, \ldots, \partial_n\}$, and $\what{\Der}$ its $\mm$-adic completion, which coincides with the $\what{\hol}$-module of $\C$-derivations of $\what{\hol}$ and is free (as an $\what{\hol}$-module) with the same basis. \\

\noindent Every derivation $\delta \in \Der$ can be uniquely decomposed as a sum $\delta = \sum_{i=-1}^\infty \delta_i$, where $\delta_i = \sum_{j=1}^n a_{ij} \partial_j$ with $a_{ij}$ being homogeneous of degree $i+1$ for all $j=1, \ldots, n$. Moreover, $\delta_0$, which is called the \emph{linear part of $\delta$}, can be written as $\underline{x} A \overline{\partial}$, where $A$ is a constant matrix and $\underline{x} = (x_1, \ldots, x_n)$. The same is true for a formal derivation $\delta \in \what{\Der}$. We say $\delta \in \Der$ (resp. $\delta \in \what{\Der}$) is \emph{singular} if $\delta_{-1} = 0$ or, equivalently, $\delta \in \mm \Der$ (resp. $\delta \in \mm \what{\Der} = \what{\mm} \what{\Der}$). \\

\noindent For a convergent power series $f \in \hol$ (or a formal power series $f \in \what{\hol}$), we will denote by $\Der_f$ the $\hol$-module ($\what{\hol}$-module) of convergent (formal) logarithmic derivations. Its $\mm$-adic completion will be denoted by $\what{\Der}_f$ and will consist of those derivations $\delta \in \what{\Der}$ such that $\delta(f) \in \what{\hol} f$. If $f \in \hol$, we can consider it as a formal power series and the formal $\Der_f$ (which is an $\what{\hol}$-module) is precisely the $\mm$-adic completion of the $\hol$-module $\Der_f$. \\

\noindent As in Definition \ref{DefFEH}, we will say that a formal power series $f \in \what{\hol}$ is \emph{strongly Euler-homogeneous at $0$} if there exists a formal (singular) derivation $\delta \in \mm \what{\Der}$ such that $\delta(f)=f$. Let us see the relation with the convergent definition. \\

\begin{Prop}
    Let $f \in \hol$. Then, $f$ is strongly Euler-homogeneous at $0$ as a convergent power series if and only if so is it as a formal power series.
\end{Prop}

\begin{proof}
    The direct implication is clear, since $\mm \Der \subset \mm \what{\Der}$ and $\hol \subset \what{\hol}$. For the other one, let us note that if $f$ is strongly Euler-homogeneous at $0$ as a formal power series, then $f \in \what{\mm} \langle \partial_i f, i = 1, \ldots, n \rangle$. But, by faithful flatness, $\hol \cap \what{\mm} \langle \partial_i f, i = 1, \ldots, n \rangle = \mm \langle \partial_i f, i = 1, \ldots, n \rangle$, so $f \in \mm \langle \partial_i f, i = 1, \ldots, n \rangle$. This means that there exists $\delta \in \mm \Der$ such that $\delta(f)=f$ and $f$ is strongly Euler-homogeneous at $0$ as a convergent power series.
    
\end{proof}

\vspace{0,3 cm}

\noindent Another important concept that will be used throughout the paper is that of \emph{product}. First, we give the geometric definition (which justifies its name) and then we establish its algebraic translation. \\

\begin{Def}
    Let $(D,x)$ be a germ of a divisor in a complex analytic manifold $X$ of dimension $n$. We will say that $(D,x)$ is a \emph{product} if there exists some germ of divisor $(D', 0) \subseteq (\C^{n-1},0)$ such that $(D,x)$ is biholomorphic to $(D',0) \times (\C,0)$. If $D = V(f)$ with $f \in \hol_{X,x}$ is a product, we will also say that \emph{$f$ is a product (at $x$)}.
\end{Def}

\vspace{0,3 cm}

\noindent The following equivalence is proved in a more general setting in \cite[Lemma 3.5]{Saito}: \\

\begin{Lema}
Let $f$ be a local equation of a divisor $D \subset X$ at a point $x \in D$. Then, $f$ is a product at $x$ if and only if it admits a non-singular logarithmic derivation, that is, if $\Der_{X,x}(-\log D) $ is not contained in $\mm_x \Der_{X,x}$. 
\end{Lema}

\vspace{0,3 cm}

\begin{Rmk}
    Note that, for $f \in \hol$ and with the previous notation, $f$ is a product (at $0$) if and only if $\Der_f \not\subset \mm \Der$. From now on, we will adopt this as our definition.
\end{Rmk}

\vspace{0,3 cm}

\begin{Def}
    Let $f \in \what{\hol}$. We say that $f$ is a \emph{product} if $\Der_f \not\subset \mm \what{\Der}$.
\end{Def}

\vspace{0,3 cm}

\noindent For a convergent power series, this definition is equivalent to the previous one: \\

\begin{Prop}
\label{PropFormProd}
    Let $f \in \hol$. Then, $f$ is a product as a convergent power series if and only if so is it as a formal power series.
\end{Prop}

\begin{proof}
    Let us suppose that $f$ is not a product as a convergent series, so that $\Der_f \subset \mm \Der$. Taking completions we get $\what{\Der}_f \subset \mm \what{\Der}$. But $\what{\Der}_f$ is the formal $\Der_f$, so $f$ is not a product as a formal series. Reciprocally, if $f$ is not formally a product, then $\Der_f \subset \what{\Der}_f \subset \mm \what{\Der}$. By faithful flatness, $\mm \what{\Der} \cap \Der = \mm \Der$, so $\Der_f \subset \mm \what{\Der} \cap \Der = \mm \Der$ and $f$ is not convergently a product.
    
\end{proof}

\vspace{0,3 cm}

\noindent The following lemma is well-known, (see for example \cite[Lemmata 7.3 and 7.4]{GS}), and it will allow us to just consider the case in which the local equation of the divisor is not a product. \\

\begin{Lema}
\label{LemProd}
Let $f \in \hol$. Then, $f$ is a product if and only if there exists a coordinate change such that $f=ug$ for some unit $u$ and some convergent power series $g$ in $n-1$ variables. In this case, freeness and strongly Euler-homogeneity are equivalent for $f$ and $g$. Moreover, the divisor defined by $f$ satisfies LCT if and only if so does the divisor defined by $g$ in $\C^{n-1}$.
\end{Lema}

\vspace{0,2 cm}

\section{Semisimple and nilpotent derivations}
\label{SecSN}

\noindent It was K. Saito who, in order to prove the equivalence between strong Euler-homogeneity and local quasihomogeneity for isolated singularities \cite[Theorem 4.1]{Saito71}, first introduced the notions of a semisimple differential operator and a niloperator. \\

\noindent Given a singular derivation $\delta$, K. Saito says that $\delta$ is \emph{semisimple} if it only has linear part and its associated matrix is diagonalizable. This definition does depend on the coordinate system and so it is not intrinsic. He says that $\delta$ is a \emph{niloperator} if the associated matrix of $\delta_0$ is nilpotent. Note that this definition does not depend on the coordinate system. \\

\noindent R. Gérard and A. Levelt introduce in \cite{GL} intrinsic definitions of semisimplicity and nilpotency and those will be the ones that will be used in our discussion. They first define the category $\mathcal{V}$ of complete filtered $k$-vector spaces as the category whose 

\begin{itemize}
    \item \textit{objects} are $k$-vector spaces $V$ equipped with a decreasing filtration $F = \{F^i V\}_{i \in \N}$ such that every vector space $V/F^i V$ is finite dimensional, $\bigcap_{i \in \N} F^i V = \{0\}$ and $V$ is complete for the topology defined by $F$ (where $F$ is a fundamental system of neighbourhoods of $0$). 
    \item \textit{morphisms} $T: V \to W$ are $k$-linear maps such that $T(F^i V) \subset F^i W$ for all $i \in \N$.
\end{itemize}

\vspace{0,3 cm}

\noindent The classical definitions of semisimple and nilpotent endomorphisms are then replaced by these new ones.

\vspace{0,2 cm}

\begin{Def}
Let $V \in \Obj(\mathcal{V})$, a morphism $T: V \to V$ is called:

\begin{itemize}
    \item \emph{semisimple} if every closed vector space $H$ stable by $T$ admits a supplementary subspace $H'$ (that is, $V = H \oplus H'$) that is also stable by $T$.
    
    \item \emph{topologically nilpotent} if for all $x \in V$, the sequence $\{T^n x\}_{n \in \N}$ tends to $0$ (that is, for all $i \in \N$ there exists $n_0 \in \N$ such that for $n \geq n_0, T^n x \in F^i V$).
\end{itemize}

\end{Def}

\vspace{0,2 cm}

\noindent They state an analogue of the classical Jordan decomposition theorem for the morphisms of this category. Then, they apply it to the case of formal power series and its singular derivations \cite[Théorèmes 1.5, 2.2 and 2.3]{GL}. Although the theorem is valid for every algebraically closed field of characteristic $0$, we will formulate it in the case $k = \C$. 

\vspace{0,3 cm}

\begin{Th}[Jordan decomposition]
\label{ThChangeCoordDiag}
Let $\delta$ be a singular derivation of $\what{\hol}$. Then, there exist two commuting singular derivations $\delta_S$, $\delta_N$ such that $\delta_S$ is semisimple, $\delta_N$ is topologically nilpotent and $\delta = \delta_S + \delta_N$. Moreover, this decomposition is unique and there exists a regular system of parameters of $\what{\hol}$ in which $\delta_S$ is diagonal (i.e. $\delta_S$ is of the form $\sum_{i=1}^n \lambda_i x_i \partial_i$) and the matrix of the linear part of $\delta_N$ is nilpotent.
\end{Th}

\vspace{0,2 cm}

\begin{Rmk}
    K. Saito already proved that there exists a coordinate system in which this decomposition holds \cite[Theorem 3.1]{Saito71}. Theorem \ref{ThChangeCoordDiag} shows that, actually, the decomposition holds intrinsically and that, in a particular coordinate system, it takes the desired form by K. Saito.
\end{Rmk}

\vspace{0,3 cm}

\noindent Let us note that a derivation $\delta \in \Der$ is singular if and only if $\delta(\mm^k) \subseteq \mm^k$ for all $k \in \N$, so they induce a $\C$-linear map $\delta^{(k)}: \hol/\mm^k \to \hol/\mm^k$ in each quotient. The same property holds for formal singular derivations $\delta \in \mm \what{\Der}$ in the quotients $\what{\hol}/\what{\mm}^k$. The following proposition relates the properties of these maps with those of $\delta$. \\

\begin{Prop}
\label{PropNilpQuot}
A derivation $\delta \in \mm\what{\Der}$ is topologically nilpotent if and only if, for every $k \in \N$, the induced $\C$-linear map $\delta^{(k)}: \what{\hol}/\what{\mm}^k \to \what{\hol}/\what{\mm}^k$ is nilpotent.
\end{Prop}

\begin{proof}

Suppose $\delta$ is topologically nilpotent. Let us fix $k \in \N$ and let $[a_1], \ldots, [a_m]$ be a basis of the $\C$-vector space $\what{\hol}/\what{\mm}^k$. By definition, for each $i \in \{1, \ldots, m\}$ there exists $n_i \in \N$ such that $\delta^{n_i} (a_i) \in \what{\mm}^k$. Let $n = \max \{n_i\}$, so $\delta^n(a_i) \in \what{\mm}^k$ for all $i = 1, \ldots, m$ and $(\delta^{(k)})^n ([a_i]) = [\delta^n(a_i)] = 0$. Since $(\delta^{(k)})^n$ is zero when applied to a basis, it must be the zero map, so $\delta^{(k)}$ is nilpotent. 

\vspace{0,2 cm}

\noindent Reciprocally, let $f \in \what{\hol}$ and $k \in \N$. Since by hypothesis $\delta^{(k)}$ is nilpotent, there exists $n_0 \in \N$ such that $(\delta^{(k)})^{n_0} = 0$. So $\delta^{n_0} (f) \in \what{\mm}^k$ and $\delta^n(f) \in \what{\mm}^k$ for all $n \geq n_0$.

\end{proof}

\vspace{0,2 cm}

\begin{Rmk}
By \cite[Proposition 1.3]{GL}, we also know that a singular derivation $\delta$ is semisimple if and only if $\delta^{(k)}$ is semisimple (or diagonalizable, since these two concepts are equivalent for finite dimensional complex vector spaces) for all $k \in \N$. In particular, if $\delta = \underline{x} A \overline{\partial}$ with $A$ semisimple (definition of semisimplicity for K. Saito), after a linear coordinate change, $A$ is diagonal and every induced map $\delta^{(k)}$ has a diagonal matrix with respect to the canonical basis, so $\delta$ is semisimple.
\end{Rmk}

\vspace{0,3 cm}

\begin{Cor}
\label{EquivNilp}
A derivation $\delta \in \mm\what{\Der}$ is topologically nilpotent if and only if the matrix of $\delta_0$ is nilpotent. In particular, singular derivations with no linear part are topologically nilpotent.
\end{Cor}

\begin{proof}
    If the matrix of $\delta_0$ is nilpotent, by \cite[Lemma 2.4]{GS}, every induced map $\delta^{(k)}$ is nilpotent and we deduce that $\delta$ is topologically nilpotent by Proposition \ref{PropNilpQuot}. Reciprocally, if $\delta$ is topologically nilpotent, then $\delta^{(2)}$ is nilpotent. The matrix of $\delta_0$, call it $A$, coincides with that of the restriction of $\delta^{(2)}$ to $\what{\mm}/\what{\mm}^2$ with respect to the basis $\{[x_1], \ldots, [x_n]\}$. 
    Since the restriction is still a nilpotent map, its associated matrix, $A$, is nilpotent.
    
\end{proof}

\begin{Rmk}
    This result states that the definitions of topologically nilpotent derivation by Gérard-Levelt and niloperator by K. Saito are equivalent.
\end{Rmk}

\vspace{0,3 cm}

\noindent Motivated by Corollary \ref{EquivNilp} and for the sake of simplicity, in what follows, we will remove the word ``topologically'' when referring to a topologically nilpotent derivation. \\

\noindent As we might expect, it is not true that if $\delta \in \mm \Der$ with $\delta_0$ semisimple, then $\delta$ is semisimple. Just consider $\delta = x \partial_x + y^2 \partial_y$. We have that $\delta_0 = x \partial_x$ is semisimple, $y^2 \partial_y$ is nilpotent and $[x \partial_x, y^2 \partial_y] = 0$, so $\delta_S = x \partial_x$ and, by uniqueness, $\delta$ cannot be semisimple. \\

\begin{Cor}
    Let $\delta$ be a singular derivation of $\what{\hol}$. If we denote by $\delta_{S,0}$ and $\delta_{N,0}$ the linear parts of $\delta_S$ and $\delta_N$, respectively, then $\delta_0 = \delta_{S,0} + \delta_{N,0}$ is the Jordan decomposition of $\delta_0$, that is, $\delta_{S,0} = \delta_{0,S}$ and $\delta_{N,0} = \delta_{0,N}$, where $\delta_{0,S}$ and $\delta_{0,N}$ are, respectively, the semisimple and nilpotent parts of $\delta_0$.
\end{Cor}

\begin{proof}
    By the previous discussion about the restriction of the induced maps to $\what{\mm}/\what{\mm}^2$, the matrices of $\delta_{S,0}$ and $\delta_{N,0}$ are, respectively, semisimple and nilpotent. Thus, $\delta_{S,0}$
    is semisimple and $\delta_{N,0}$ is nilpotent. Besides, $[\delta_{S,0}, \delta_{N,0}] = [\delta_S, \delta_N]_0 = 0$ since $\delta_S$ and $\delta_N$ commute. So $\delta_0 = (\delta_S + \delta_N)_0 = \delta_{S,0} + \delta_{N,0}$ is written as the sum of a semisimple and a nilpotent derivation that commute. By uniqueness, this must be the Jordan decomposition of $\delta_0$.
    
\end{proof}

\vspace{0,3 cm}

\begin{Cor}
\label{CorLogN}
    Let $f \in \what{\hol}$, $f \neq 0$, and let $\delta \in \mm \what{\Der}$ be a nilpotent derivation such that $\delta(f) = bf$ for some $b \in \what{\hol}$. Then $b \in \what{\mm}$.
\end{Cor}

\begin{proof}
    Since $\delta$ is singular, $\delta(f) \in \what{\mm}$, so if $f$ is a unit, then it is clear that $b \in \what{\mm}$. If $f \in \what{\mm}$, let $k$ be such that $f \in \what{\mm}^k \setminus \what{\mm}^{k+1}$. It is easy to see by induction that $\delta^{n} (f) \equiv b_0^n f (\hspace{-0,2 cm}\mod \what{\mm}^{k+1})$ for all $n \in \N$, where $b_0$ is the constant coefficient of $b$. By Proposition \ref{PropNilpQuot}, $\delta^{(k+1)}$ is nilpotent, so there exists $n$ with $(\delta^{(k+1)})^n = 0$. But then

    $$ 0 = (\delta^{(k+1)})^n([f]) = [\delta^n(f)] = b_0^n [f].$$

    \vspace{0,2 cm}

    \noindent Since $f \notin \what{\mm}^{k+1}$, $[f] \neq 0$, so $b_0 = 0$ and $b \in \what{\mm}$.
    
\end{proof}

\vspace{0,3 cm}

\noindent Let us prove that the decomposition of a singular derivation into a semisimple and a nilpotent part preserves the property of being logarithmic. This seems very natural, but we have not found it in the literature.

\vspace{0,3 cm}

\begin{Prop}
\label{PropSNLog}
Let $f \in \what{\hol}$ and let $\delta$ be a singular logarithmic derivation for $f$. Then $\delta_S$ and $\delta_N$ as in Theorem \ref{ThChangeCoordDiag} are also (singular) logarithmic derivations for $f$.
\end{Prop}

\begin{proof}
Since $\delta = \delta_S + \delta_N$, we have that the induced map in $\what{\hol}/\what{\mm}^k$ is $\delta^{(k)} = \delta_S^{(k)} + \delta_N^{(k)}$. We know that $\delta_S^{(k)}$ is semisimple, $\delta_N^{(k)}$ is nilpotent (by the previous proposition) and $[\delta_S^{(k)}, \delta_N^{(k)}] = [\delta_S, \delta_N]^{(k)} = 0^{(k)}$, so $\delta_S^{(k)}$ and $\delta_N^{(k)}$ are precisely the semisimple and nilpotent part of $\delta^{(k)}$. Each quotient $\what{\hol}/\what{\mm}^k$ is a $\C$-vector space of finite dimension, so $\delta_S^{(k)}$ and $\delta_N^{(k)}$ can be written as polynomials in $\delta^{(k)}$ by the Jordan-Chevalley decomposition \cite[Cap. VII, Th. 1]{Bourbaki}. But since $\delta(f) = af$ with $a \in \what{\hol}$, $\delta^{(k)}([f]) = [a][f]$, and taking into account that $\delta^{(k)}$ is also a derivation, $(\delta^{(k)})^n([f]) \in \langle [f] \rangle$ for all $n$. So a polynomial in $\delta^{(k)}$ maps $[f]$ to an element in $\langle [f] \rangle$ and we conclude that $\delta_S^{(k)}([f]), \delta_N^{(k)}([f]) \in  \langle [f] \rangle$. This means that $\delta_S(f), \delta_N(f) \in \langle f \rangle + \what{\mm}^k$ for all $k \in \N$. But $\bigcap_{k \geq 1} \left( \langle f \rangle  + \what{\mm}^k\right) = \langle f \rangle$ since every ideal in a Noetherian local ring is closed for the $\what{\mm}$-adic topology \cite[Theorem 8.14]{Matsumura}, so $\delta_S(f), \delta_N(f) \in \langle f \rangle$ and $\delta_S, \delta_N$ are logarithmic.

\end{proof}

\begin{Cor}
\label{CorBasis}
If $f \in \hol$ is free and $\Der_f \subset \mm \Der$ (that is, every logarithmic derivation is singular), then there exists a basis of $\what{\Der}_f$ formed by semisimple and nilpotent derivations. Moreover, $f$ is a formal unit multiplied by the determinant of the Saito matrix with respect to that basis.
\end{Cor}

\begin{proof}
Since $f$ is free, $\Der_f \subset \mm \Der$ is free as an $\hol$-module and $\what{\Der}_f \subset \mm \what{\Der}$ is free as an $\what{\hol}$-module. Let $\{\delta_1, \ldots, \delta_n\}$ be an $\what{\hol}$-basis of $\what{\Der}_f$. By Proposition \ref{PropSNLog}, $\{\delta_{1,S}, \delta_{1, N},  \ldots, \delta_{n,S}, \delta_{n,N}\}$ is a generating set of $\what{\Der}_f$ and, by Nakayama's Lemma, we can extract a basis $\mathcal{C}$ of $\what{\Der}_f$ from this set. For the second part, if $A$ is the Saito matrix with respect to a convergent basis $\mathfrak{B}$ and $B$ is the change-of-basis matrix from $\mathcal{C}$ to $\mathfrak{B}$ (whose columns are the coordinates of the elements of $\mathcal{C}$ with respect to $\mathfrak{B}$), then $C = B^t A$ is the Saito matrix of $\mathcal{C}$. By Saito's Criterion, there exists a unit $u$ such that $f = u \det(A)$. Since $B$ is invertible, $\det(B)$ is a (formal) unit, so $\det(A) = \det(B)^{-1} \det(C)$ and $f = u \det(B)^{-1} \det(C)$, where $u \det(B)^{-1}$ is a formal unit.

\end{proof}

\begin{Rmk}
    Although we will only use the preceding result in the case of free divisors, the first assertion is also true for non-free divisors if we replace the word ``basis'' by ``generating set''.
\end{Rmk}

\section{A necessary condition for LCT in free divisors}
\label{SecTr0}

\noindent In \cite{1996}, the authors show, by constructing two different double complexes and comparing the four associated spectral sequences, that if LCT holds for a free divisor $D$ in $\C^n$, then, for each basis $\{\delta_1, \ldots, \delta_n\}$ of $\Der(-\log D)$, the morphism

\vspace{0,2 cm}

$$ \begin{array}{cccc}
     d_1: & \check{H}^{n-1} (V \setminus \{0\}, \hol_X) & \to & \check{H}^{n-1} (V \setminus \{0\}, \hol_X)^n \\
          & [g] & \mapsto & ([\delta_1(g)], \ldots, [\delta_n(g)])
     
\end{array} $$

\vspace{0,2 cm}

\noindent must be injective, where $\check{H}^{n-1} (V \setminus \{0\}, \hol_X)$ denotes \v{C}ech cohomology of a sufficiently small Stein neighbourhood $V$ of $0$ with values in $\hol_X$ with respect to the open cover $\{V_i = \{x_i \neq 0\}, i=1, \ldots, n\}$. It can be identified with the set of Laurent series with strictly negative powers in every variable. \\

\noindent We define \emph{the trace of a singular derivation $\delta$}, and denote it $\tr(\delta)$, as the trace of its linear part. Note that this definition is valid for both convergent and formal derivations. Moreover, the trace of a singular convergent derivation coincides with the trace of its corresponding derivation in the completion, since $\mm/\mm^2 \cong \what{\mm}/\what{\mm}^2$.

\vspace{0,3 cm}

\noindent It is easy to see \cite[Lemma 7.5]{GS} that, if $\delta$ is singular, then 

\vspace{0,2 cm}

$$ \delta \left( \left[ \frac{1}{x_1 \cdot \ldots \cdot x_n} \right] \right) = \left[ \frac{\tr(\delta)}{x_1 \cdot \ldots \cdot x_n} \right]. $$

\vspace{0,3 cm}

\noindent So if LCT holds for $D$ and it is not a product at $0$, then some of the $\delta_i$ must have non-zero trace, as otherwise the class of $\prod_{i=1}^n x_i^{-1}$ would be in the kernel of $d_1$. In this section, we will show that we can reach the same conclusion in an alternative way by using $\D$-module theory. \\

\noindent Let $\D = \hol[\partial_1, \ldots,\partial_n]$ be the ring of linear differential operators and consider the left ideal $\D \langle x_1, \ldots, x_n \rangle$. The following lemma describes the singular derivations lying in this ideal.

\vspace{0,5 cm}

\begin{Lema}
\label{LemaIdealTr0}
A singular derivation $\eta = \sum_{i=1}^n a_i \partial_i$ belongs to the left ideal $\D \langle x_1, \ldots, x_n \rangle$ if and only if $\tr(\eta) = 0$.
\end{Lema}

\begin{proof}
Let $c_i = \dfrac{\partial a_i}{\partial x_i}(0)$, that is, the coefficient of the term $x_i$ in $a_i$, so that $\tr(\eta) = \sum_{i=1}^n c_i$. By the relation $[\partial_i, a_i] := \partial_i \cdot a_i - a_i \partial_i = \partial_i(a_i)$ and taking into account that $\partial_i \cdot a_i \in \D \langle x_1, \ldots, x_n \rangle$ since $a_i \in \mm$, we have that $\eta = -\sum_{i=1}^n c_i$ modulo $\D \langle x_1, \ldots, x_n \rangle$. Thus, $\eta$ will belong to $\D \langle x_1, \ldots, x_n \rangle$ if and only if $\tr(\eta) = \sum_{i=1}^n c_i = 0$.

\end{proof}

\noindent In \cite[Corollary 4.2]{Cald2005}, Calderón and Narváez proved that the following statements are equivalent:

\begin{itemize}
    \item[(1)] $D$ satisfies LCT.
    \item[(2)] The natural morphism $\rho: \D \stackrel{\mathbb{L}}{\otimes}_{\mathcal{V}_0} \hol(D) \to \hol(*D)$ induced by the inclusion $\hol(D) \subset \hol(* D)$ is an isomorphism in the derived category of bounded complexes of coherent left $\D$-modules, where $\mathcal{V}_0 = \D(-\log D)$ denotes the ring of logarithmic differential operators, $\hol(*D)$ is the $\hol$-module of meromorphic functions with poles along $D$ and $\hol(D)$ is the submodule of meromorphic functions with poles of order at most $1$.
\end{itemize}

\vspace{0,3 cm}

\noindent We are now able to prove the desired result.

\vspace{0,3 cm}

\begin{Th}
\label{Tr0}
If $D$ is a free divisor in $\C^n$ satisfying LCT and $f \in \hol$ is a reduced local equation of $D$ at $0$ such that it is not a product, then there exist derivations of $\Der_f$ with non-zero trace.
\end{Th}

\begin{proof}

\vspace{0,3 cm}

\noindent Consider a basis $\mathfrak{B} = \{\delta_1, \ldots, \delta_n\}$ of $\Der_f \subset \mm \Der$ and suppose that every derivation of $\mathfrak{B}$ has zero trace. By Lemma \ref{LemaIdealTr0}, $\D \langle \delta_1, \ldots, \delta_n \rangle \subset \D \langle x_1, \ldots, x_n \rangle$ and we have a surjective morphism of holonomic $\D$-modules \\

$$ \frac{\D}{\D \langle \delta_1, \ldots, \delta_n \rangle} \longrightarrow \frac{\D}{\D \langle x_1, \ldots, x_n \rangle} \longrightarrow 0. $$ \\

\noindent By applying the duality functor $\mathbb{D}(-) = \R \HHom_{\D}(-, \D)[n]^{\leftfun}$ we get an injective morphism \\

$$ 0 \longrightarrow \mathbb{D} \left( \frac{\D}{\D\langle x_1, \ldots, x_n \rangle} \right) \longrightarrow \mathbb{D} \left(\frac{\D}{\D \langle \delta_1, \ldots, \delta_n \rangle} \right). $$ \\

\noindent It is a well-known fact that $\dfrac{\D}{\D\langle x_1, \ldots, x_n \rangle}$ is self-dual (that is, its dual is isomorphic to itself) and that $\D \stackrel{\mathbb{L}}{\otimes}_{\mathcal{V}_0} \hol \cong \mathbb{D} \left(\D \stackrel{\mathbb{L}}{\otimes}_{\mathcal{V}_0} \hol(D) \right)$ \cite[Corollary 3.1.2]{Cald2005}. Thus, if $D$ satisfies LCT, then $\D \stackrel{\mathbb{L}}{\otimes}_{\mathcal{V}_0} \hol \cong \mathbb{D} \left( \hol(*D) \right)$, meaning $\D \stackrel{\mathbb{L}}{\otimes}_{\mathcal{V}_0} \hol$ is concentrated in degree $0$ and

$$ \D \stackrel{\mathbb{L}}{\otimes}_{\mathcal{V}_0} \hol = \D \otimes_{\mathcal{V}_0} \hol \cong \D \otimes_{\mathcal{V}_0} \dfrac{\mathcal{V}_0}{\mathcal{V}_0 \langle \delta_1, \ldots, \delta_n \rangle} \cong \dfrac{\D}{\D\langle \delta_1, \ldots, \delta_n \rangle}. $$

\vspace{0,2 cm}

\noindent Applying $\mathbb{D}$ to both sides we get

\vspace{0,2 cm}

$$ \mathbb{D} \left(\frac{\D}{\D \langle \delta_1, \ldots, \delta_n \rangle} \right) \cong \mathbb{D} \left( \D \stackrel{\mathbb{L}}{\otimes}_{\mathcal{V}_0} \hol  \right) \cong \hol(*D). $$

\vspace{0,5 cm}

\noindent Finally, we have an injective morphism

\vspace{0,2 cm}

$$ 0 \longrightarrow \frac{\D}{\D \langle x_1, \ldots, x_n \rangle} \longrightarrow \hol(*D). $$

\vspace{0,3 cm}

\noindent But $\D / \D \langle x_1, \ldots, x_n \rangle$ cannot be a submodule of $\hol(*D)$ since the first one has $\hol$-torsion and the second one does not, so we conclude that at least one derivation of $\mathfrak{B}$ must have non-zero trace.

\end{proof}

\noindent A similar statement holds even if we remove the condition of not being a product:

\vspace{0,2 cm}

\begin{Cor}
    If $D$ is a free divisor in $\C^n$ satisfying LCT and $f \in \hol$ is a reduced local equation of $D$ at $0$, then there exist singular derivations of $\Der_f$ with non-zero trace.
\end{Cor}

\begin{proof}
    The case in which $f$ is not a product is already proved in Theorem \ref{Tr0}, so we may assume that $f$ is a product. Then, applying inductively Lemma \ref{LemProd}, we can write $f = ug$ for a certain unit $u$ and a convergent power series $g$ in $m < n$ variables such that it is not a product. Since LCT also holds for the divisor $D'$ defined by $g$ in $\C^m$, by Theorem \ref{Tr0}, we can find singular logarithmic derivations for $g$ with non-zero trace. These derivations are also logarithmic for $f$, so we are done.
    
\end{proof}

\section{LCT and Euler-homogenenity for free divisors}
\label{SecLCTEH}

\vspace{0,3 cm}

We begin this section with a technical lemma that is valid in an arbitrary number of variables, although we will only use it in the cases $n=2$ and $n=3$.

\vspace{0,3 cm}

\begin{Lema}
\label{DiagDer}
Let $\delta = \sum_{i=1}^n \lambda_i x_i \partial_i$ be a logarithmic diagonal derivation for a formal equation $f \in \what{\hol}$ such that $\delta(f) = cf$ for some $c \in \what{\hol}$. Then, there exist a formal unit $u$ and $g \in \what{\hol}$ such that $f=ug$ and $\delta(g)=c_0 g$, where $c_0$ is the value of $c$ at the origin.
\end{Lema}

\begin{proof}

Let us write $c = \sum_{{{\alpha}}} c_{{\alpha}} {x}^{{\alpha}}$, where $\alpha = (\alpha_1, \ldots, \alpha_n)$ runs in $\N^n$ and $x^\alpha = x_1^{\alpha_1}\cdot \ldots \cdot x_n^{\alpha_n}$. Note that if ${\alpha}$ verifies ${\alpha \cdot \lambda} := \sum_{j=1}^n \alpha_j \lambda_j \neq 0$ and we set

$$ a_{{\alpha}} = \exp \left( \frac{c_{{\alpha}} {x}^{{\alpha}}}{{\alpha \cdot \lambda}} \right),$$

\vspace{0,3 cm}

\noindent then $a_\alpha$ verifies $\delta(a_{{\alpha}}) = c_{{\alpha}} {x}^{{\alpha}} \cdot a_{{\alpha}}$. \\

\noindent Now, set

$$ u = \prod_{{\alpha \cdot \lambda} \neq 0} a_{{\alpha}} = \exp \left( \sum_{{\alpha \cdot \lambda} \neq 0} \frac{c_{{\alpha}} x^{{\alpha}}}{{\alpha \cdot \lambda}}  \right).  $$

\vspace{0,3 cm}

\noindent Note that $u$ is a well-defined formal power series because $b := \sum_{{\alpha \cdot \lambda} \neq 0} \frac{c_{{\alpha}} x^{{\alpha}}}{{\alpha \cdot \lambda}}$ vanishes at the origin, so we can compose the two series and the result is another formal power series that does not vanish at the origin. Thus, $u$ is a formal unit and

$$\delta(u) = \delta(e^b) = e^b \delta(b) = u \cdot \sum_{{\alpha \cdot \lambda} \neq 0} c_{{\alpha}} x^{{\alpha}}.$$

\noindent Then, we can write $f = u g$ with $g = u^{-1} f$ and

$$ \delta(g) = \delta(u^{-1}) f + u^{-1} \delta(f) = -u^{-1} \cdot \left( \sum_{{\alpha \cdot \lambda} \neq 0} c_{{\alpha}} x^{{\alpha}} \right) f + u^{-1} c f = \left( \sum_{{\alpha \cdot \lambda} = 0} c_{{\alpha}} x^{{\alpha}} \right) g. $$

\vspace{0,3 cm}

\noindent Let us see that $c' := \sum_{{\alpha \cdot \lambda} = 0} c_{{\alpha}} x^{{\alpha}}$ must be $c_0$ and we will have the result:

\vspace{0,3 cm}

\noindent Note that $\delta(c') = 0$, so if we decompose $g$ as a sum of eigenvectors of $\delta$ (what can be done in a unique way by \cite[Lemma 2.3]{Saito71}), $g = \sum_{\mu \in \C} g_{\mu}$ with $\delta(g_\mu) = \mu g_\mu$, we have

\vspace{0,2 cm}

$$ \delta(g) = \sum_{\mu \in \C} \mu g_\mu = c' g = \sum_{\mu \in \C} c' g_\mu, $$

\vspace{0,3 cm}

\noindent and since $\mu g_\mu$ and $c'g_\mu$ are both eigenvectors of $\delta$ for $\mu$, by the uniqueness of the decomposition, we must have $\mu g_\mu = c' g_\mu$ for every $\mu \in \C$. Since $g \neq 0$ (as otherwise $f$ is zero and there is nothing to prove), some $g_\mu \neq 0$. But then $c' = \mu$ and the only possibility for this is if $c' = \mu = c_0$, as the value of $c'$ at the origin is $c_0$.

\end{proof}

\begin{Rmk}
The same result is true for any semisimple derivation since there exists a formal coordinate change in which it is diagonal. Particularly, for any $\delta \in \mm \Der$, if $\delta_S(f) \in \mm f$, then there exists a formal unit $u$ and a formal power series $g$ such that $f=ug$ and $\delta_S(g) = 0$.
\end{Rmk}

\vspace{0,3 cm}

\begin{Prop}
\label{PropMN1}
Let $D$ be a free divisor in a complex analytic manifold of dimension $n$. Let $f$ be a reduced local equation of $D$ around $0$ such that it is not a product. If $D$ is not strongly Euler-homogeneous at $0$, then $f \in \mm^{n+1}$.
\end{Prop}

\begin{proof}
 Let $\delta_1, \ldots, \delta_n$ be a basis of $\Der_f$ and $\alpha_i \in \hol$ be such that $\delta_i(f) = \alpha_i f$ for $i=1, \ldots, n$. We may assume that $f = \det(A)$ with $A$ the Saito matrix, so $\overline{\alpha} f = \overline{\delta}(f) = A \overline{\partial}(f)$ and $\overline{\partial}(f) = \Adj(A) \overline{\alpha}$, where $\overline{\alpha} = (\alpha_1, \ldots, \alpha_n)^t$ and $\Adj(A)$ is the adjugate matrix of $A$ (the transpose of its cofactor matrix). Since $D$ is not strongly Euler-homogeneous at $0$, every $\alpha_i \in \mm$ (as otherwise $\delta_i/\alpha_i$ would be an Euler vector field for $f$). But $f$ is not a product, so $\Der_f \subset \mm \Der$ and all the entries of $A$ also lie in $\mm$. This implies that the entries of $\Adj(A)$ lie in $\mm^{n-1}$ and hence the coordinates of $\overline{\partial}(f) = \Adj(A) \overline{\alpha}$ lie in $\mm^n$. Thus, the partial derivatives of $f$ lie in $\mm^n$ and $f$ lies in $\mm^{n+1}$.
 
\end{proof}

\vspace{0,3 cm}

\noindent In what follows, $f$ will be a reduced local equation around $0$ of a plane curve $D \subset \C^2$ that it is not a product. \\

\begin{Prop}
\label{PropN1}
If $f \in \mm^3$ and $\delta$ is a singular logarithmic derivation such that $\delta(f) \in \mm f$, then $\delta$ is nilpotent.
\end{Prop}

\begin{proof}
Let $\delta = \delta_S + \delta_N$ be the decomposition of $\delta$ into a semisimple and a nilpotent part that commute and choose a (formal) coordinate system for which $\delta_S$ is diagonal. We know that $\delta_S$ and $\delta_N$ are also logarithmic for $f$, so there exist $a, b \in \what{\hol}$ such that $\delta_S(f) = af$ and $\delta_N(f) = bf$. Since $\delta(f) \in \mm f$ and $b \in \mm$ by Corollary \ref{CorLogN}, we have that $a \in \mm$. Now, by the preceding lemma, there exists a unit $u$ and a series $g$ such that $f=ug$ and $\delta_S(g) = 0$.

\vspace{0,3 cm}

\noindent Let us suppose that $\delta$ is not nilpotent and so $\delta_S := \lambda_1 x \partial_x + \lambda_2 y \partial_y \neq 0$. In particular, $(\lambda_1, \lambda_2) \neq (0,0)$, and since $\delta_S(g) = 0$, every monomial $x^{\alpha_1} y^{\alpha_2}$ appearing in $g$ must verify $\alpha_1 \lambda_1 + \alpha_2 \lambda_2 = 0$. Thus, the initial form of $g$ (note that $g \neq 0$ since $f \neq 0$) is, up to a constant, $x^p y^q$ for some $p,q \in \N$ and every other monomial of $g$ will be of the form $x^{kp} y^{kq}$ for some $k \in \N$. This means we can write $g = v \cdot x^p y^q$ for a unit $v$. But then, $f = uv x^p y^q \in \mm^3$, so $p+q \geq 3$ and $f$ is not reduced as a formal power series. This implies by Proposition \ref{PropConvFormalRed} that $f$ is not reduced as a convergent series and we get a contradiction, so we must have $\delta_S = 0$ and $\delta = \delta_N$ is nilpotent.

\end{proof}

\vspace{0,3 cm}

\noindent Every plane curve is a free divisor. Now, the proof of the conjecture for $n=2$ is an almost straightforward consequence of the preceding proposition.

\vspace{0,3 cm}

\begin{Cor}
\label{CorConjN2}
Plane curves that satisfy LCT are strongly Euler-homogeneous.
\end{Cor}

\begin{proof}
Since these are local properties, it is enough to show that if $f$ is not strongly Euler-homogeneous at $0$, then LCT does not hold for $f$ at $0$. We may assume that $f$ is not a product, as otherwise the problem would be reduced by Lemma \ref{LemProd} to dimension 1, where the equation becomes $u \cdot x_1 = 0$ with $u$ a unit and $x_1 \partial_1$ is a strong Euler vector field for $x_1$. So $\Der_f \subseteq \mm \Der$ and $f = \det(A) \in \mm^2$ with $A$ the Saito matrix for a suitable basis of $\Der_f$. If $f \notin \mm^3$, then we know that it is strongly Euler-homogeneous at $0$ by Proposition \ref{PropMN1} and there is nothing to prove, so we may assume $f \in \mm^3$.

\vspace{0,3 cm}

\noindent Let us suppose that $f$ is not strongly Euler-homogeneous at $0$, so $\delta(f) \in \mm f$ for all $\delta \in \Der_f$. This means, by the preceding proposition, that every derivation on $\Der_f$ is nilpotent. In particular, by Corollary \ref{EquivNilp}, the matrices of their linear parts are nilpotent. But then all their traces are $0$ and, by Theorem \ref{Tr0}, LCT cannot hold.

\end{proof}

\vspace{0,3 cm}

\noindent Unfortunately, Proposition \ref{PropN1} cannot be generalized to higher dimension. Indeed, it is not even true for $n \geq 3$ that if $f \in \mm^{n+1}$ and $\delta(f) \in \mm f$, then $\tr(\delta) = 0$, as the following example shows: \\

\begin{Ex}
Let $n \geq 3$ and $f = (x_1^3-x_2^3) x_3 \ldots x_n = 0 \in \mm^{n+1}$. $\Der_f$ is a free module of which a basis is $\{ \delta_1, \ldots, \delta_n \}$, where 

$$\begin{array}{l}
     \delta_1 = x_1 \partial_1 + x_2 \partial_2, \\[0,2 cm]
     \delta_2 = x_2^2 \partial_1 + x_1^2 \partial_2, \\[0,2 cm]
     \delta_i = x_i \partial_i, \hspace{0,2 cm} 3 \leq i \leq n.
\end{array}$$

\vspace{0,3 cm}

\noindent Let $\eta_i = 3 \delta_i - \delta_1$, $3 \leq i \leq n$. We have that $\eta_i$ is diagonal (and therefore, not nilpotent) and $\eta_i(f) = 0$ but $\tr(\eta_i) = 1$ for all $3 \leq i \leq n$.

\end{Ex}

\vspace{0,5 cm}

\noindent Nevertheless, Lemma \ref{DiagDer} allows us to simplify the proof of Conjecture \ref{ConjLCT} in dimension 3 given in \cite{GS}, without using the \emph{formal structure theorem} \cite[Theorem 5.4]{GS}. But first, we need to introduce the concept of a \emph{Koszul free divisor} \cite[Definition 1.6]{CN02}.

\vspace{0,3 cm}

\begin{Def}
Let $D \subset X$ be a divisor in a complex analytic manifold of dimension $n$. We say that $D$ is \emph{Koszul free} at $x \in X$ if it is free at $x$ and there exists a basis $\{\delta_1, \ldots, \delta_n\}$ of $\Der_{X,x}(-\log D)$ such that the sequence of symbols $\{\sigma(\delta_1), \ldots, \sigma(\delta_n)\}$ is regular in the graded ring with respect to the order filtration of the ring of linear differential operators at $x$. We say $D$ is Koszul free if it is Koszul free at each $x \in X$.
\end{Def}

\vspace{0,3 cm}

\noindent In \cite[Corollary 1.5]{GS}, the authors prove that Koszul free divisors satisfy Conjecture \ref{ConjLCT}, so we can restrict ourselves to the case in which $D$ is not Koszul free.

\vspace{0,3cm}

\begin{Th}
Let $D$ be a free divisor in a complex analytic manifold of dimension 3. If LCT holds for $D$, then it is strongly Euler-homogeneous.
\end{Th}

\begin{proof}
Let us suppose that $D$ is not strongly Euler-homogeneous. As before, we may assume that $D$ is not a product (in such a case, by Lemma \ref{LemProd}, we can reduce the problem to dimension 2, where we know the result is true) and that it is not strongly Euler-homogeneous at $0$, so if $f$ is a reduced local equation for $D$ at $0$, then we have $\Der_f \subseteq \mm \Der$ and $\delta(f) \in \mm f$ for all $\delta \in \Der_f$. We may also assume that $D$ is not Koszul free at $0$, as we know the result is true for Koszul free divisors. Let us see that every $\delta \in \what{\Der}_f$ must verify $\tr(\delta) = 0$ (by linearity, it is enough to show it for a basis). In particular, each $\delta$ in $\Der_f$ will also have zero trace, and by Theorem \ref{Tr0} this will imply that LCT cannot hold for $D$. \\

\noindent By Corollary \ref{CorBasis}, we can take a basis $\mathfrak{B}$ of $\what{\Der}_f$ in which each derivation is either semisimple or nilpotent. Let $A$ be the Saito matrix of this basis, so $f = u \cdot \det(A)$ for some unit $u \in \what{\hol}$. As we discussed previously, every nilpotent derivation has zero trace, so it remains to prove that the same is true for the semisimple derivations of $\mathfrak{B}$. \\

\noindent Let $\delta$ be a semisimple derivation of $\mathfrak{B}$. By Theorem \ref{ThChangeCoordDiag}, there exists a formal system of coordinates in which $\delta$ is diagonal (say $\delta = \sum_{i=1}^3 \lambda_i x_i \partial_i$). Since $\delta(f) \in \mm f$, by Lemma \ref{DiagDer}, there exist a unit $v \in \what{\hol}$ and $g \in \what{\hol}$ such that $f = vg$ and $\delta(g) = 0$ (note that $g$ is also reduced by Proposition \ref{PropConvFormalRed}) Now, we can distinguish three cases (which are the subcases of Case II in the proof of Theorem 1.6 in \cite{GS}):

\begin{itemize}
    \item Only one of the $\lambda_i$ is not zero: after renaming the variables we may assume $\lambda_1 = \lambda_2 = 0, \lambda_3 \neq 0$. In this case, $g$ is annihilated by $\partial_3$, which is not in $\mm \what{\Der}$, so $D$ is (formally, and then, convergently by Proposition \ref{PropFormProd}) a product and we get a contradiction.
    
    \item Exactly two of the $\lambda_i$ are not zero: as before, we may assume $\lambda_1 = 0, \lambda_2, \lambda_3 \neq 0$. In this case, $g$ is of the form $\sum_{\lambda_2 j + \lambda_3 k = 0} a_{jk}(x_1) x_2^j x_3^k$ or, equivalently,  $\sum_{\mu \in \N} a_{\mu}(x_1) x_2^{\mu p} x_3^{\mu q}$ where $p, q \geq 1$ are such that $\lambda_2 p + \lambda_3 q = 0$ (if $p=q=0$, $g$ would be a product). Since $g = v^{-1} u \det(A)$ and $\lambda_1 = 0$, $g \in \langle x_2, x_3 \rangle$. This implies that $a_0(x_1) = 0$ and we can extract $x_2^p x_3^q$ as common factor. But $g$ is reduced, so it must be $p=q=1$ and $\lambda_2 + \lambda_3 = 0$, meaning $\tr(\delta) = 0$.
    
    \item None of the $\lambda_i$ is zero: if we truncate the coordinate change that makes $\delta$ diagonal at a sufficiently high order, we get a convergent $\eta \in \Der_f$ such that its linear part is precisely of the form $\sum_{i=1}^3 \lambda_i y_i \partial_i$. In this case, $L_X(\log D)$, the logarithmic characteristic subvariety, which is the variety defined by the symbols of the logarithmic derivations in the cotangent bundle, has dimension $n=3$. Indeed, as the singular locus of a (singular) free divisor is of the largest possible dimension, we have that $\Sing D$ is of dimension $1$. But since $\lambda_i \neq 0$ for all $i$, there exists a sufficiently small neighbourhood $V \subset D$ in which $\eta$ only vanishes at the origin, so $\sigma(\eta)$ determines a hypersurface (which is of dimension $2$) over each singular point of $V \setminus\{0\}$. Thus, the maximal possible dimension of $L_X(\log D)$ is $1+2 = 3$, which is also the minimal one \cite[3.17]{Saito}. But then $D$ is Koszul free by \cite[Corollary 1.9]{CN02} and we get a contradiction.
\end{itemize}

\noindent We conclude that every derivation of $\mathfrak{B}$ must have zero trace, and so LCT cannot hold.

\end{proof}

\vspace{0,5 cm}

\begin{center}
    {\bf APPENDIX}
\end{center}

\appendix

\section{Clarifying the argument in the proof of the conjecture for $n=2$ in \cite{2002}}
\label{Sec2002}

\noindent In \cite[Theorem 3.3]{2002}, the authors state without proof that if $D$ is a plane curve and $f \in \mm^3$ is a reduced local equation for $D$ around $0$, then there exists a basis $\{\delta_1, \delta_2\}$ of $\Der_f$ such that $\delta_1$ has no linear part. Although this is true, the proof needs some attention.

\vspace{0,3 cm}

\noindent What can be said a priori is the following (later we will prove that, in fact, condition $(2)$ cannot hold).

\vspace{0,3 cm}

\begin{Prop}
\label{Prop1}
Let $D$ be a plane curve and let $f \in \mm^3$ be a reduced local equation for $D$ around $0$. Then, after a possible change of variables, there exists a basis of $\Der_f$ (as an $\hol$-module) such that one of the following two conditions holds:

\begin{itemize}
    \item[(1)] At least one of its elements has zero linear part (and therefore it is nilpotent).
    \item[(2)] The linear parts of its elements are $x\partial_x$ and $(ax+y)\partial_x$ for some $a \in \C$.
\end{itemize}

\end{Prop}

\begin{proof}
Every plane curve is a free divisor, so by Saito's Criterion, we know that there exists a basis $\{\delta_1, \delta_2\}$ of $\Der_f$ such that

$$ \left( \begin{array}{c} \delta_1 \\ \delta_2 
\end{array} \right) = A  \left( \begin{array}{c} \partial_x \\ \partial_y \end{array} \right)$$

\vspace{0,2 cm}

\noindent with $f = \det(A)$. 

\vspace{0,2 cm}

\noindent Note that, since $f$ is reduced and belongs to $\mm^3$, it cannot be a product by Lemma \ref{LemProd}. Hence, all the entries of $A$ lie in $\mm$ and the determinant of the linear parts of the entries of $A$ must vanish. If $\delta_{i0} = (a_i x + b_i y) \partial_x + (c_i x + d_i y) \partial_y$ for $i = 1,2$, then we have

$$ \left| \begin{array}{cc} a_1 & c_1 \\ a_2 & c_2 \end{array} \right| = \left| \begin{array}{cc} b_1 & d_1 \\ b_2 & d_2 \end{array} \right| = \left| \begin{array}{cc} a_1 & d_1 \\ a_2 & d_2 \end{array} \right| + \left| \begin{array}{cc} b_1 & c_1 \\ b_2 & c_2 \end{array} \right| = 0. $$

\vspace{0,2 cm}

\noindent Let us consider the derivations $\delta_1' = a_2 \delta_1 - a_1 \delta_2$ and $\delta_2' = -b_2 \delta_1 + b_1 \delta_2$, whose linear parts are:
    
    $$ \begin{array}{c}
         \delta'_{10} = (a_2 b_1 - a_1 b_2) y \partial_x + (a_2 d_1 - a_1 d_2) y \partial_y, \\
         \delta'_{20} = (a_2 b_1 - a_1 b_2) x \partial_x + (a_2 d_1 - a_1 d_2) x \partial_y,
    \end{array} $$
    
\noindent where we have used that $b_1 c_2 - b_2 c_1 = a_2 d_1 - a_1 d_2$. There are two possibilities:
    
\begin{itemize}

\item If $a_2 b_1 - a_1 b_2 \neq 0$, then $\{\delta'_1, \delta'_2\}$ is another basis of $\Der_f$. Dividing by $a_2 b_1 - a_1 b_2$ and setting $a = (a_2 d_1 - a_1 d_2)/(a_2 b_1 - a_1 b_2)$ we have derivations whose linear parts are $x(\partial_x + a \partial_y)$ and $y(\partial_x + a \partial_y)$. By performing the linear coordinate change $x' = x$ and $y' = y-ax$ we have $\partial_x = \partial_{x'}-a\partial_{y'}$ and $\partial_{y} = \partial_{y'}$, so we get $(2)$ in the variables $x', y'$.
        
\item If $a_2 b_1 - a_1 b_2 = 0$, we claim that also $a_1 d_2 - a_2 d_1 = 0$. It is clear if $a_1 = a_2 = 0$, and otherwise $(b_1, b_2) = \lambda (a_1, a_2)$ for some $\lambda \in \C$. From $0 = b_1 d_2 - b_2 d_1 = \lambda (a_1 d_2 - a_2 d_1)$ we get $a_1 d_2 - a_2 d_1 = 0$ or $\lambda = 0$, in which case $(b_1, b_2) = (0,0)$ and we get $a_1 d_2 - a_2 d_1 = 0$ from the last equality. Hence, $\delta'_{10}$ and $\delta'_{20}$ are zero. If at least one of $a_1, a_2, b_1, b_2$ is not zero, then $\delta'_1$ or $\delta'_2$ can be part of a basis of $\Der_f$ and we have $(1)$.

Otherwise, we have $\delta_{10} = (c_1 x + d_1 y) \partial_y$ and $\delta_{20} = (c_2 x + d_2 y) \partial_y$, so we can consider $\delta'_1 = c_2 \delta_1 - c_1 \delta_2$ and $\delta'_2 = -d_2 \delta_1 + d_1 \delta_2$, whose linear parts are:

    $$ \begin{array}{c}
         \delta'_{10} = (c_2 d_1 - c_1 d_2) y \partial_y, \\
         \delta'_{20} = (c_2 d_1 - c_1 d_2) x \partial_y.
    \end{array} $$

Again, if $c_2 d_1 - c_1 d_2 \neq 0$, we have that $\{\delta'_1, \delta'_2\}$ is a basis of $\Der_f$. Dividing by $c_2 d_1 - c_1 d_2$ and exchanging $x$ and $y$, we have a basis with $x \partial_x$ and $y \partial_x$ as linear parts. This is the second case with $a = 0$. 

Finally, if $c_2 d_1 - c_1 d_2 = 0$, then, either $c_1 = d_1 = 0$ or $(c_2, d_2) = \eta (c_1, d_1)$ for some $\eta \in \C$. In the first case we have $\delta_{10} = 0$ and in the second one, we have that $\{\delta'_1 = \delta_1, \delta'_2 = \eta \delta_1 - \delta_2 \}$ is a basis of $\Der_f$ with $\delta'_{20} = 0$.
\end{itemize}
\end{proof}

\vspace{0,3 cm}

\noindent The following lemma gives a necessary condition for a plane curve to be strongly Euler-homogeneous in terms of the linear parts of a basis of its logarithmic derivations.

\vspace{0,3 cm}

\begin{Lema}
Let $D$ be a plane curve and let $f \in \mm^3$ be a reduced local equation for $D$ around $0$ such that it is not a product. Let $\{\delta_1, \delta_2\}$ be a basis of $\Der_f$ with $\delta_{i0} = (a_i x + b_i y) \partial_x + (c_i x + d_i y) \partial_y$ for $i = 1,2$ and let $B = \left( \begin{array}{cccc} a_1 & b_1 & c_1 & d_1 \\ a_2 & b_2 & c_2 & d_2 \end{array} \right)$. If $D$ is strongly Euler-homogeneous, then $\rank(B) = 1$.
\end{Lema}

\begin{proof}
Let $\alpha_i \in \hol$ be such that $\delta_i(f) = \alpha_i f$. We may assume that $f = \det(A)$ with $A$ the Saito matrix ($\overline{\delta} = A \overline{\partial}$), so $\overline{\alpha} f = A \overline{\partial}(f)$ and $\overline{\partial}(f) = \Adj(A) \overline{\alpha}$. In particular, we have

$$ \begin{array}{r}
   \partial_x(f) = a_{22} \alpha_1 - a_{12} \alpha_2  \\
   \partial_y(f) = -a_{21} \alpha_1 + a_{11} \alpha_2 
\end{array} \hspace{0,2 cm} \text{ if } 
A = \left(\begin{array}{cc}
    a_{11} & a_{12} \\
    a_{21} & a_{22}
\end{array} \right). $$

\vspace{0,2 cm}

\noindent Since $f \in \mm^3$, $\partial_x(f)$ and $\partial_y(f)$ are in $\mm^2$, so if we call $\alpha_{i0}$ the constant term of $\alpha_i$:

$$ \left\lbrace \begin{array}{r}
     (c_2 x + d_2 y) \alpha_{10} - (c_1 x + d_1 y) \alpha_{20} = 0  \\
     -(a_2 x + b_2 y) \alpha_{10} + (a_1 x + b_1 y) \alpha_{20} = 0 
\end{array} \right. $$

\vspace{0,2 cm}

\noindent From this, we get the homogeneous linear system of equations

$$ \left( \begin{array}{cc}
    -a_2 & a_1 \\
    -b_2 & b_1 \\
     -c_2 & c_1 \\
     -d_2 & d_1 \\
\end{array} \right) 
\left( \begin{array}{c}
     \alpha_{10} \\
     \alpha_{20} 
     \end{array} \right) = \mathbf{0}. $$

\vspace{0,2 cm}

\noindent Note that the rank of the matrix of this system is the rank of $B$. Since $D$ is strongly Euler-homogeneous, some $\alpha_i$ must be a unit, and so this system must have a non-trivial solution, what can only happen if $\rank(B) \leq 1$. But $B$ cannot be the zero matrix, as otherwise every $\delta \in \Der_f$, not having linear term, would increase the order of $f$ and none of them could be an Euler vector field. We conclude that $\rank(B) = 1$.

\end{proof}

\vspace{0,3 cm}

\begin{Cor}
    In the situation of Proposition \ref{Prop1}, $(2)$ cannot hold (and so the statement given in the proof of Theorem 3.3 in \cite{2002} is true).
\end{Cor}

\begin{proof}
    Let us suppose that $(2)$ holds and let $\{\delta_1, \delta_2\}$ be the basis of $\Der_f$ as in Proposition \ref{Prop1}. The matrix $B$ of the preceding lemma is in this case
    
    $$ B = \left( \begin{array}{cccc} 1 & 0 & 0 & 0 \\ a & 1 & 0 & 0 \end{array} \right), $$

    \vspace{0,2 cm}

    \noindent which has rank $2$ for every $a \in \C$, so $D$ is not strongly Euler-homogeneous. This means that $\delta_i(f) \in \mm f$ for $i=1,2$. Then, the linear parts of the basis, as they do not increase the order, must annihilate the initial form of $f$. So if $f_{k} \in \mm^k \setminus \mm^{k+1}$ ($k \geq 3$) is the initial form of $f$, then $\partial_x(f_{k}) = 0$ and $f_{k} = y^k$. But in the proof of Theorem 3.3 in \cite{2002}, the authors show that, if $f_{k} = x^p y^q$ and the linear part of a logarithmic derivation that is not an Euler vector field is $qx \partial_x -py \partial_y$, then $f = x^p y^q$ after a coordinate change. We have a particular case of this fact with $p=0$ and $q=k$, so $f$ can be written as $y^k$ in some coordinate system. But then $f$ is not reduced, contradicting the initial supposition.
    
\end{proof}

\vspace{0,1 cm}

\begin{Rmk}
Apart from this, there is no need of considering that one of the logarithmic derivations has no linear part in the proof of Theorem 3.3 in \cite{2002}. What can be easily deduced from the arguments given there is that every $\delta \in \Der_f$ such that $\delta(f) \in \mm f$ must be nilpotent, as we proved in a completely different way in Proposition \ref{PropN1}. Then, if $D$ is not strongly Euler-homogeneous, every logarithmic derivation is nilpotent and, by the argument given in Corollary \ref{CorConjN2}, LCT cannot hold.
\end{Rmk}

\nocite{*}
\bibliographystyle{plainurl}
\bibliography{biblio}

\end{document}